\documentclass[aip,jmp,amsfonts,amssymb,amsmath]{article}

\usepackage{graphicx}
\usepackage{dcolumn}
\usepackage{bm}
\usepackage{color}
\usepackage{amsfonts,amssymb,amsmath}
\usepackage{exscale,cmmib57}

\newcommand{\R}{\mathbb{R}}

\newcommand{\test}{C^\infty_0(D)}

\newcommand{\be}{\begin{equation}}
\newcommand{\ee}{\end{equation}}

\newtheorem{theorem}{Theorem}

\newtheorem{remark}[theorem]{{Remark}}
\newtheorem{lemma}[theorem]{Lemma}
\newcommand*{\QEDB}{\hfill\ensuremath{\square}}
\newenvironment{proof}{\noindent {\it Proof.}}{\QEDB\smallskip}

\begin{document}

\newenvironment{example}{\noindent {\bf Example.}}{\hfill{\tiny{$\blacksquare$}}}

\title{A well-posed variational formulation\\ of the Neumann boundary value problem\\ for the biharmonic operator}

\author{Alberto Valli \\ {} \\Department of Mathematics, University of Trento, Italy}

%\date{\today}

%\pacs{}

\maketitle

\begin{abstract}  
In this note we devise and analyze a well-posed variational formulation of the Neumann boundary value problem associated to the biharmonic operator $\Delta^2$. An alternative formulation as a system of two Poisson problems for the Laplace operator $\Delta$ is also derived.
\end{abstract}

\section{Introduction}\label{intro}
Let $D \subset \R^d$, $d \ge 2$, be a bounded, connected, open set, with Lipschitz boundary $\partial D$; the unit outward normal vector on $\partial D$ is indicated by $n$. We denote by $H^s(D)$, $s \ge 0$, the usual Sobolev spaces of Hilbert type (i.e., when the summability exponent $p$ is equal to $2$). 

We want to analyze the Neumann boundary value problem for the biharmonic operator $\Delta^2$, namely, the problem associated to the differential boundary operators $(\Delta u)_{|\partial D}$ and $(\nabla \Delta u \cdot n)_{|\partial D}$. 

We are aware that this problem is not the most interesting from the physical point of view, as the model which describes the equilibrium position of an elastic thin plate, unconstrained on the boundary, involves other second order and third order boundary operators, in which the Poisson ratio also has a role (see Courant and Hilbert \cite[p.\ 250]{couhil} and more recently, e.g., Verchota \cite{ver}, Provenzano \cite{prov}; the original physical model even dates back to Kirchhoff and Kelvin). 

Let us also mention that, adopting another point of view in which the operator $\Delta^2$ is not decomposed as $\Delta \Delta$ but by means of two different second order operators, Pauly and Schomburg \cite{PauSch} (see also Pauly and Zulehner \cite{PauZul}) have  analyzed the biharmonic operator with second order and third order differential boundary conditions which are different from those of the Neumann problem and are more possible to have a physical interpretation.

However, despite these remarks, we think that the Neumann problem for the biharmonic operator has a nice and simple mathematical structure, similar to that of other classical problems, and we find it interesting from the mathematical point of view. Moreover, it is the limiting case, for the Poisson ratio going to $1$, of suitable physical models.  

In the following ``an inductive approach is favored, sometimes at the expense of the conciseness which can be gained by a deductive, authoritarian mode of presentation" (borrowed from Courant and Hilbert \cite[p.\ vii]{couhil2}).

\subsection{A preliminary comparison: the Neumann problems for $-\Delta$ and for $\Delta^2$}
As it is well-known, the Neumann boundary value problem for the Laplace operator $-\Delta$ 
\begin{equation}
\left\{\begin{array}{ll}
- \Delta \varphi = \sigma &\hbox{\rm in} \ D\\ 
(\nabla \varphi \cdot n)_{|\partial D} = \tau &\hbox{\rm on} \ \partial D
\end{array}
\right.
\end{equation}
and what we call the Neumann boundary value problem for the biharmonic operator $\Delta^2$ 
\begin{equation}\label{neubih}
\left\{\begin{array}{ll}
\Delta^2 u = f &\hbox{\rm in} \ D\\ 
(\Delta u)_{|\partial D} = g &\hbox{\rm on} \ \partial D \\
(\nabla \Delta u \cdot n)_{|\partial D} = h &\hbox{\rm on} \ \partial D
\end{array}
\right.
\end{equation}
have similarities and differences.

In particular, for both of them the solution is not unique: adding to $\varphi$ a constant and adding to $u$ a harmonic function gives another solution. Moreover, for both of them the data have to satisfy a compatibility condition:  for the operator $-\Delta$ we have, by integrating by parts, 
$$
\int_D \sigma \, dx = - \int_D \Delta \varphi \, dx = - \int_{\partial D} \nabla \varphi \cdot n \, dS_x =  - \int_{\partial D} \tau \, dS_x \ ,
$$
while for the operator $\Delta^2$ it holds
\begin{equation}\label{compcon}
\begin{array}{ll}
\int_D f \, \eta \, dx = \int_D \Delta^2 u \, \eta \, dx &= \int_{\partial D} \nabla \Delta u \cdot n \, \eta \, dS_x - \int_{\partial D} \Delta u \, \nabla \eta \cdot n \, dS_x \\ [5pt]
&\qquad \qquad + \int_D \Delta u \, \Delta \eta \, dx \\ [5pt]  
&= \int_{\partial D} h \, \eta \, dS_x - \int_{\partial D} g \, \nabla \eta \cdot n \, dS_x \end{array}
\end{equation}
for any $\eta$ such that $\Delta \eta = 0$ in $D$.

A first difference concerns the type of the boundary conditions: the Neumann boundary condition for the $-\Delta$ operator satisfies the so-called complementing condition of Agmon, Douglis and Nirenberg \cite{ADN}, that in the present situation simply says that the polynomial $B_1(t) = t$ is not divisible by $t-i$.

This is not true for the Neumann boundary condition for the $\Delta^2$ operator: the complementing condition would require that the polynomials $B_1(t) = 1+t^2$ and $B_2(t)= t+t^3$ were linearly independent modulo $(t-i)^2$, and it is well-known that this is not the case.
For the ease of the reader, let us readily show this last statement: we have
$$
\frac{B_1(t)}{(t-i)^2} = 1 + \frac{2(1+it)}{(t-i)^2} \ \ , \ \ \frac{B_2(t)}{(t-i)^2} = t + 2i + \frac{2(-t+i)}{(t-i)^2} \, ,
$$
and the second remainder $2(-t+i)$ is proportional to the first remainder $2(1+it)$ by a factor $i$.

Note that the difference between the two Neumann problems still shows up when considering the so-called Lopatinski\u{\i}--\v{S}apiro condition (see Wloka \cite[Sect. 11, Examples 11.2 and 11.8]{Wloka}); in fact, it is known that the complementing condition and the Lopatinski\u{\i}--\v{S}apiro condition are equivalent (see, e.g., Negr\'on-Marrero and Montes-Pizarro \cite[Appendix A]{NegMon}).

Another difference seems to be related to well-posedness: in fact, the Neumann boundary value problem associated to the $-\Delta$ operator is well-posed, in a suitable space where uniqueness is recovered and for data which satisfy the necessary compatibility condition. This well-known result is an easy consequence of the Poincar\'e inequality and the Lax--Milgram theorem in $H^1_*(D)= \{v\in H^1(D) \, | \, \int_D v \, dx = 0\}$ (see, e.g., Girault and Raviart \cite[Sect. 1.4]{GirRav}, Valli \cite[Chap. 5]{AlbV}). 

On the contrary, well-posedness for the Neumann boundary value problem associated to the $\Delta^2$ operator seems to be questionable (see, e.g., what is explicitly reported in Verchota \cite[p. 217 and Sect. 21]{ver}, and in a more indirect way in Renardy and Rogers \cite[Sect. 9.4.2 and Example 9.30]{RenRog}, Gazzola, Grunau and Sweers \cite[Sect. 2.3]{GGS}, Provenzano \cite[p. 1006]{prov}). In addition to this, it can be noted that Begehr \cite{beg} presents a  long list of boundary value problems (twelve!) for the biharmonic operator that are either well-posed or solvable under suitable compatibility conditions, and in that list the Neumann problem is not included. 

Going a little bit more in depth, in Renardy and Rogers \cite{RenRog}, Gazzola, Grunau and Sweers \cite{GGS} and Provenzano \cite{prov} the comments about the fact that the Neumann problem for the biharmonic operator $\Delta^2$ is possibly not well-posed are related to the fact that the complementing condition is not satisfied (in particular, this condition is assumed in the existence and uniqueness Theorems 2.16 and 2.20 in \cite{GGS}; there see also Remark 2.17). This seems to be meaningful, as in Agmon, Douglis and Nirenberg \cite[Sect. 10]{ADN} it is explicitly proved that the complementing condition is necessary for obtaining higher order a-priori estimates in H\"older and $L^p$ spaces (in this respect, see also Lions and Magenes \cite[Chap. 2, Sect. 8.3 and Remark 9.8]{LioMag}). 

However, rather surprisingly, it turns out that this condition is not necessary for well-posedness in suitable Hilbert spaces, as this example shows.

\begin{example}
Consider the operator
$$
\Delta^2 - \Delta + \hbox{\sl Id} \, ,
$$
with the boundary operators
$$
(\Delta u)_{|\partial D} \ \ , \ \ \ (\nabla \Delta u \cdot n)_{|\partial D} - (\nabla u \cdot n)_{|\partial D} \, .
$$
Since the complementing condition only depends on the principal parts of the spatial and boundary operators, we are in the same situation of the Neumann problem for the biharmonic operator; therefore the complementing condition is not satisfied.
However, the weak formulation of the problem (with vanishing boundary data) reads
\be\label{exam}
\begin{array}{ll}
\text{find} \ u \in W \ &: \ \int_D \Delta u \, \Delta v \, dx + \int_D \nabla u \cdot \nabla v \, dx + \int_D u\, v \, dx \\ [5pt]
&\qquad \qquad = \int_D f \, v \, dx \quad \forall \  v \in W \, ,
\end{array}
\ee
where $W=\{v \in H^1(D) \, | \, \Delta v \in L^2(D)\}$, which is a Hilbert space with respect to the scalar product 
$$
(\cdot,\cdot)_W= \int_D \Delta u \, \Delta v \, dx + \int_D \nabla u \cdot \nabla v \, dx + \int_D u\, v \, dx 
$$
(clearly, the Laplace operator being intended in the usual weak sense). Therefore for each $f \in L^2(D)$ the existence and uniqueness of the solution to \eqref{exam} follows directly from the Riesz representation theorem, and the stability estimate $\|u\|_W \le C \|f\|_{L^2(D)}$ also holds.
\end{example}

Being now evident that the complementing condition is not necessary for well-posedness in a suitable Hilbert space, in the next section we go to show that indeed the Neumann boundary value problem for the biharmonic operator is well-posed.

\section{The Neumann boundary value problem for the biharmonic operator is well-posed}
The fact that the Neumann boundary value problem for the biharmonic operator has a solution can be formally shown in a quite simple way (at least for the case $g=0$ and $h=0$).
Set ${\cal H}=\{\eta \in L^2(D) \, | \, \Delta \eta = 0 \ \text{in} \ D\}$, which is a closed subspace of $L^2(D)$. Suppose that $f = \Delta q$ with $q \in H^2_0(D)$; it is easy to verify that under this assumption the necessary condition $f \in {\cal H}^\bot$ 
is satisfied (here orthogonality has to be intended in the $L^2(D)$-sense; to check the result, just integrate by parts twice in $\int_D f \, \eta \, dx=\int_D \Delta q \, \eta \, dx$, where $\eta \in {\cal H}$, and use the boundary conditions $q_{|\partial D} = 0$ and $(\nabla q \cdot n)_{|\partial D} = 0$). Being available this function $q$ one solves the Dirichlet boundary value problem for the biharmonic operator
$$
\left\{\begin{array}{ll}
\Delta^2 w = q &\hbox{\rm in} \ D\\ 
w_{|\partial D} = 0 &\hbox{\rm on} \ \partial D \\
(\nabla w \cdot n)_{|\partial D} = 0 &\hbox{\rm on} \ \partial D 
\end{array}
\right.
$$
and takes $u = \Delta w$. It readily follows  $\Delta^2 u = \Delta(\Delta^2 w) = \Delta q = f$, and $(\Delta u)_{|\partial D} = (\Delta^2 w)_{|\partial D} = q_{|\partial D} = 0$, $(\nabla \Delta u \cdot n)_{|\partial D} = (\nabla \Delta^2 w \cdot n)_{|\partial D} = (\nabla q \cdot n)_{|\partial D} = 0$.

Whether the assumption $f = \Delta q$, $q \in H^2_0(D)$, is equivalent to the necessary condition $f \in {\cal H}^\bot$ will be addressed later (see Theorem \ref{overth}).

\subsection{A variational formulation and its analysis}
We want to devise a suitable variational formulation and analyze it in a careful and complete way. For carrying this approach at its end we will need to prove some results that seem to be interesting by themselves (see Lemma \ref{close}, Lemma \ref{V=X}, Lemma \ref{trace} and Remark \ref{immerH2}).

We assume $f \in L^2(D)$, $g \in L^2(\partial D)$ and $h \in L^2(\partial D)$ (weaker assumptions on $g$ and $h$ could be done: see Theorem \ref{1stexth}). By (formal) integration by parts we trivially have
\begin{equation}\label{intbpbharm}
\begin{array}{ll}
&\int_D f \, v \,dx = \int_D (\Delta^2 u) \, v \, dx = -\int_D \nabla \Delta u \cdot \nabla  v \, dx + \int_{\partial D} \nabla \Delta u \cdot {n} \, v \, dS_x \\ [4pt]
& \qquad  = \int_D \Delta u \, \Delta v \, dx - \int_{\partial D} \Delta u \, \nabla v \cdot n \, dS_x + \int_{\partial D} \nabla \Delta u \cdot {n} \, v \, dS_x \\ [4pt]
& \qquad  = \int_D \Delta u \, \Delta v \, dx - \int_{\partial D} g \, \nabla v \cdot n \, dS_x + \int_{\partial D}h \, v \, dS_x\, .
\end{array}
\end{equation}
Therefore the variational formulation we are looking for is at first identified with the choices:
$$
\begin{array}{ll}
V = L^2(\Delta;D) = \{v \in L^2(D) \, | \, \Delta v \in L^2(D)\} \\ [4pt]
B(w,v) = \int_D \Delta w \, \Delta v \, dx \\ [4pt]{F}(v) =  \int_D f \, v \, dx + \int_{\partial D} g \, \nabla v \cdot {n} \, dS_x - \int_{\partial D} h \,  v  \, dS_x  \, .
\end{array}
$$
We recall that in the definition of $V$ the Laplace operator is intended in the standard weak sense: for $v \in L^1_{\text{loc}}(D)$ we say that a function $q \in L^1_{\text{loc}}(D)$ is the weak Laplacian of $v$  if
$$
\int_D q \, \varphi \ dx = \int_D v \, \Delta \varphi \, dx \quad \forall \ \varphi \in \test \, .
$$ 
Thus it is easily seen that $L^2(\Delta;D)$ is a Hilbert space with respect to the natural scalar product $\int_D(w\,v + \Delta w \, \Delta v)\, dx$.

However, it is clear that with this choice of the variational space the problem cannot be well-posed (uniqueness does not hold: adding to a solution a harmonic function belonging to $L^2(D)$ gives another solution). Therefore the space $V$ should be replaced by a closed subspace of it which does not contain non-zero harmonic functions.  We define
$$
V_\sharp =  L^2(\Delta;D) \cap {\cal H}^\bot  \, .
$$

Note also that the definition of the linear operator $F(\cdot)$ is not completely clear: having only assumed $g\in L^2(\partial D)$ and $h\in L^2(\partial D)$, the boundary integrals would require $v_{|\partial D} \in L^2(\partial D)$ and $(\nabla v \cdot {n})_{|\partial D} \in L^2(\partial D)$, and for a function $v$ belonging to $L^2(\Delta;D)$ this is not always the case. A suitable trace theory for $v_{|\partial D}$ and $(\nabla v \cdot {n})_{|\partial D}$ is needed (see Lemma \ref{trace}): but for the moment leave that apart and proceed in a formal way.

In order to use the Lax--Milgram theorem we need a Poincar\' e-type inequality like
$$
\|v\|_{L^2(D)} \le C \|\Delta v\|_{L^2(D)}
$$
for all $v \in V_\sharp$. This inequality would follow, by a standard ``reductio ad absurdum", if the immersion $L^2(\Delta;D) \hookrightarrow L^2(D)$ was compact; but this result is not true (see, e.g., Valli \cite[Exercise 6.11]{AlbV}). Then we could try to show that the immersion of $V_\sharp \hookrightarrow L^2(D)$ is compact. This result is more elusive, and for the moment we do not insist on it (but see Remark \ref{immerH2}).

A different attempt can be done by changing the variational space. We introduce
$$
X =\{\omega = \Delta r \, | \, r \in H^4(D) \cap H^2_0(D)\} \, .
$$ 
We verify at once that $X \subset H^2(D)$;
moreover, we have 
\begin{lemma}\label{close}
Assume that the boundary $\partial D$ is smooth, say, of class $C^4$. Then $X$ is closed in $H^2(D)$ with respect to the $H^2(D)$-norm (thus $X$ is a Hilbert space with the $H^2(D)$-scalar product).
\end{lemma}
\begin{proof} 
In fact, if $\omega_k = \Delta r_k \to \omega$ in $H^2(D)$, we take the solution $r \in H^2_0(D)$ of the biharmonic problem $\Delta^2 r = \Delta \omega \in L^2(D)$ with the homogeneous Dirichlet boundary conditions, namely, $r_{|\partial D} = 0$ and $(\nabla r \cdot n)_{|\partial D} = 0$ (the existence and uniqueness of this solution can be found in Gazzola, Grunau and Sweers \cite[Theor. 2.15]{GGS}). 
From the regularity results for higher order elliptic equations 
we obtain
$$
\|r_k - r\|_{H^4(D)} \le C \|\Delta \omega_k - \Delta \omega\|_{L^2(D)}
$$
(see Gazzola, Grunau and Sweers \cite[Corollary 2.21]{GGS}). Since $\omega_k \to \omega$ in $H^2(D)$, it follows $r_k \to r$ in $H^4(D)$, and consequently $\Delta r_k = \omega_k \to \Delta r$ in $H^2(D)$, thus $\Delta r = \omega$ and $X$ is closed.
\end{proof}

On the other hand, the estimate above also says that for each $\omega \in X$ it holds
$$
\|\omega\|_{H^2(D)} = \|\Delta r\|_{H^2(D)} \le C \|r\|_{H^4(D)} \le C \|\Delta \omega\|_{L^2(D)} \, ,
$$
therefore the bilinear form $B(\upsilon,\omega) = \int_D \Delta \upsilon \, \Delta \omega \, dx$ is coercive in $X$. Note also that for $\omega \in X \subset H^2(D)$ the trace values $\omega_{|\partial D}$ and $(\nabla \omega \cdot n)_{|\partial D}$ have a meaning in $H^{3/2}(\partial D)$ and $H^{1/2}(\partial D)$, respectively. Summarizing, by using the Lax--Milgram theorem we have proved:

\begin{theorem}\label{1stexth}
Assume that $D \subset \mathbb{R}^{n}$ is a bounded, connected, open set, with  boundary $\partial D \in C^4$, and that $f \in L^2(D)$, $g \in H^{-1/2}(\partial D)$ and $h \in H^{-3/2} (\partial D)$. Then there exists a unique solution of the variational problem
\begin{equation}\label{weekeqbih}
\begin{array}{ll}
&\text{find } u \in X \, \\ [4pt]
&\qquad \int_D \Delta u \, \Delta \omega \, dx = \int_D f \, \omega \, dx  + \langle g,\nabla \omega \cdot {n}\rangle_{-\frac12}- \langle h,\omega\rangle_{-\frac32} \\ [4pt]
&\text{for all} \ \omega \in X \, .
\end{array}
\end{equation}
\end{theorem}
Here $\langle \cdot,\cdot\rangle_{-s}$ denotes the duality pairing between $H^{-s}(\partial D) = (H^s(\partial D))'$ and $H^{s}(\partial D)$, $s > 0$.
Let us also remark that for $\omega = \Delta r \in X$
we have
$$
\int_D \omega \eta \, dx = \int_D \Delta r \, \eta \, dx = \int_D r \, \Delta \eta \, dx = 0
$$ 
for each $\eta \in {\cal H}$, the integration by parts being justified by a density argument as $r \in H^2_0(D)$. Therefore $X \subset V_\sharp$.

\subsection{Interpretation of the result}\label{interpr}
We have thus proved the existence and uniqueness of a weak solution $u \in X \subset V_\sharp$ of a variational problem related to the same bilinear form and the same linear functional which describe the Neumann boundary value problem for the biharmonic operator. Moreover, the variational space $X$ does not contain any explicit constraint on the values on the boundary. 

However, the proof that the solution we have found is the solution of the Neumann boundary value problem for the biharmonic operator needs some work.

We start trying to determine which equation satisfies $u$ when we use test functions belonging to $H^2(D)$ (instead of $X$; note that $\test$ is contained in $H^2(D)$, while it is not contained in $X$). We can use the $L^2(D)$-orthogonal projection $P_{\cal H}$ on ${\cal H}$.  Take $q \in H^2(D)$ and set $\pi=q - P_{\cal H}q$: clearly, $\Delta \pi = \Delta q$, thus
$
\int_D \Delta u \Delta q \, dx = \int_D \Delta u \Delta \pi \, dx$.
Moreover $\pi \in {\cal H}^\bot$, thus
$\pi \in V_\sharp$. 

It is now useful the following lemma:
\begin{lemma}\label{V=X}
Assume that $D \subset \mathbb{R}^{n}$ is a bounded, connected, open set, with  $\partial D \in C^4$. Then $X = V_\sharp$, and the norms $\|\cdot\|_X$ and $\|\cdot\|_{V_\sharp}$ are equivalent.
\end{lemma}
\begin{proof}
Here above we have already seen that $X \subset V_\sharp$. Let us prove the opposite inclusion. Take $\omega \in V_\sharp$ and solve $\Delta^2 r = \Delta \omega \in L^2(D)$ with $r_{|\partial D} =0$ and $(\nabla r \cdot n)_{|\partial D} =0$ (namely, find a solution $r \in H^2_0(D)$). We have already seen that this is possible and that, by the regularity results for the biharmonic operator and provided that $\partial D \in C^4$, we obtain a unique solution $r \in H^4(D) \cap H^2_0(D)$ with the estimate $\|r\|_{H^4(D)} \le C \|\Delta \omega\|_{L^2(D)}$. Thus we have $(\Delta r - \omega) \in {\cal H}$. Moreover, for each $\eta \in {\cal  H}$,
$$
\int_D \Delta r \, \eta \, dx = \int_D r \, \Delta \eta \, dx = 0 \, ,
$$
due to the boundary conditions $r_{|\partial D} =0$ and $(\nabla r \cdot n)_{|\partial D} =0$; hence $\Delta r \in {\cal H}^\bot$ and also $(\Delta r - \omega)\in {\cal H}^\bot$. Having already seen $(\Delta r - \omega) \in {\cal H}$, we obtain $\omega = \Delta r$, $r \in H^4(D) \cap H^2_0(D)$, therefore $\omega \in X$.

Finally, we have
$$
\|\omega\|_{X} = \|\omega\|_{H^2(D)} =  \|\Delta r\|_{H^2(D)} \le C \|r\|_{H^4(D)} \le C \|\Delta \omega\|_{L^2(D)} \le C \|\omega\|_{V_\sharp} \, ,
$$
and also
$$
\|\omega\|^2_{V_\sharp} = \|\omega\|^2_{L^2(D)} + \|\Delta \omega\|^2_{L^2(D)} \le C \|\omega\|^2_{H^2(D)} \, .
$$
which ends the proof.
\end{proof}

Thus we can proceed with the interpretation of the weak problem \eqref{weekeqbih}, as now we know that $\pi=(q - P_{\cal H}q) \in X= V_\sharp$. Note also that $P_{\cal H}q \in H^2(D)$, as $q \in H^2(D)$ and $\pi \in X \subset H^2(D)$, thus $(P_{\cal H}q)_{|\partial D} \in H^{3/2}(\partial D)$ and $(\nabla P_{\cal H}q \cdot n)_{|\partial D} \in H^{1/2}(\partial D)$.
By inserting $\pi$ in the variational problem \eqref{weekeqbih} we easily find
$$
\begin{array}{ll}
\int_D \Delta u \Delta q \, dx &= \int_D \Delta u \Delta \pi \, dx = \int_D f \pi \, dx + \langle g,\nabla \pi \cdot {n}\rangle_{-\frac12} - \langle h,\pi\rangle_{-\frac32} \\[4pt]
&= \int_D f  q \, dx + \langle g,\nabla q \cdot {n}\rangle_{-\frac12} -\langle h,q\rangle_{-\frac32} \\ [4pt]
&\qquad \quad-  \int_D f  \, P_{\cal H}q \, dx - \langle g,\nabla P_{\cal H}q \cdot {n}\rangle_{-\frac12} +\langle h,P_{\cal H}q\rangle_{-\frac32} \, . 
\end{array}
$$
Since a necessary condition for determining a solution of the Neumann problem is 
\begin{equation}\label{neccon}
\int_D f  \, \eta \, dx + \langle g,\nabla \eta \cdot {n}\rangle_{-\frac12} - \langle h,\eta\rangle_{-\frac32} = 0
\end{equation}
for each $\eta \in {\cal H} \cap H^2(D)$ (see \eqref{compcon} and take into account that now we are only assuming $g \in H^{-1/2}(\partial D)$ and $h \in H^{-3/2} (\partial D)$), we conclude that the solution $u \in X$ we have found solves \eqref{weekeqbih} also for each $q \in H^2(D)$.

Before going on, let us clarify the meaning of the traces $v_{|\partial D}$ and $(\nabla v \cdot n)_{|\partial D}$ for $v \in L^2(\Delta;D)$.

\begin{lemma}\label{trace}
Assume that $D \subset \mathbb{R}^{n}$ is a bounded, connected, open set, with  boundary $\partial D \in C^2$. Then for any $v \in L^2(\Delta;D)$ one has $v_{|\partial D} \in H^{-1/2}(\partial D)$, $(\nabla v \cdot n)_{|\partial D} \in H^{-3/2}(\partial D)$. Moreover, the maps $v \to v_{|\partial D}$ and $v \mapsto (\nabla v \cdot n)_{|\partial D}$ are continuous from $ L^2(\Delta;D)$ to  $H^{-1/2}(\partial D)$ and $H^{-3/2}(\partial D)$, respectively.
\end{lemma}
\begin{proof}
This result is obtained by standard arguments, and we present the proof for the ease of the reader (see also Lions and Magenes \cite[Chap. 2, Theor. 6.5 and Sect. 9.8]{LioMag}). We define $v_{|\partial D} \in H^{-1/2}(\partial D)$ by setting
$$
\langle v_{|\partial D},\xi\rangle_{-\frac12} :=  \int_D v \, \Delta E_\xi \, dx- \int_D \Delta v\,  E_\xi \, dx \ \  \ \forall \ \xi \in H^{1/2}(\partial D) \, ,
$$
where $E_\xi \in H^2(D)$ is any extension in $D$ of $\xi$ such that $(\nabla E_\xi \cdot n)_{|\partial D} = \xi$ on $\partial D$, $E_{\xi|\partial D} = 0$  on $\partial D$ and the map $\xi \mapsto E_\xi$ is continuous from $H^{1/2}(\partial D)$ to $H^2(D)$ (for instance, the solution $w_\xi$ of the biharmonic problem $\Delta^2 w_\xi = 0$ with $w_{\xi|\partial D} = 0$ and $(\nabla w_\xi \cdot n)_{|\partial D} = \xi$). The definition does not depend on the choice of the extension; in fact, if we take two extensions $E^{(1)}_\xi$ and $E^{(2)}_\xi$, their difference $\sigma_\xi$ belongs to $H^2_0(D)$ and therefore $\int_D v \, \Delta \sigma_\xi \, dx  = \int_D \Delta v \, \sigma_\xi \, dx$ for each $v \in L^2(\Delta;D)$. Note also that the map $v \mapsto v_{|\partial D}$ is clearly continuous from $L^2(\Delta;D)$ to $H^{-1/2}(\partial D)$.

With a similar argument,  we define the first order trace $(\nabla v \cdot n)_{|\partial D}  \in H^{-3/2}(\partial D)$  as
$$
\langle (\nabla v \cdot n)_{|\partial D},\mu \rangle_{-\frac32} := \int_D \Delta v\,  F_\mu \, dx - \int_D v \, \Delta F_\mu \, dx \ \ \ \forall \ \mu \in H^{3/2}(\partial D) \, ,
$$
where $F_\mu \in H^2(D)$ is any (continuous) extension in $D$ of $\mu$ such that $F_{\mu|\partial D} = \mu$ and $(\nabla F_\mu \cdot n)_{|\partial D} = 0$ on $\partial D$. 
\end{proof}

Now let us continue the interpretation of the weak problem \eqref{weekeqbih}. We know that we can choose any test function $q \in H^2(D)$.
Selecting $q \in \test$ in \eqref{weekeqbih} we first obtain $\Delta^2 u = f$ in $D$ (in the sense of distributions). 

The latter result says that $\Delta u \in L^2(\Delta;D)$, thus by Lemma \ref{trace} one has $(\Delta u)_{|\partial D}   \in H^{-1/2}(\partial D)$ and $(\nabla  \Delta u \cdot n)_{|\partial D} \in H^{-3/2}(\partial D)$ (which is coherent with the assumptions
$g \in H^{-1/2}(\partial D)$ and $h \in H^{-3/2}(\partial D)$).

Take now $q \in H^2(D)$ and integrate by parts:
$$
\begin{array}{ll}
 \int_D \Delta u \, \Delta q \, dx  &=  - \int_D \nabla  \Delta u \cdot \nabla q \, dx +\langle \Delta u,\nabla q \cdot n\rangle_{-\frac12} \\ [4pt]
&= \int_D \Delta^2 u \, q \, dx - \langle \nabla  \Delta u \cdot n,q\rangle_{-\frac32} +\langle \Delta u,\nabla q \cdot n\rangle_{-\frac12} \, .
\end{array} 
$$
Taking into account that $\Delta^2 u = f$ in $D$, from  \eqref{weekeqbih} it follows
$$
-\langle \nabla  \Delta u \cdot n - h,q\rangle_{-\frac32} +\langle \Delta u - g,\nabla q \cdot n\rangle_{-\frac12}  \ \ \forall \ q \in H^2(D) \, .
$$ 
We must now select $q \in H^2(D)$ in a suitable way; precisely, it will be the solution $\rho \in H^2(D)$ of the Dirichlet boundary value problem $\Delta^2 \rho = 0$ in $D$ with $\rho_{|\partial D} = p_1$, $(\nabla \rho \cdot n)_{|\partial D} = p_2$, with arbitrary $p_1 \in H^{3/2}(\partial D)$ and $p_2 \in H^{1/2}(\partial D)$. This solution $\rho$ exists and is unique (see Gazzola, Grunau and Sweers \cite[Theorem 2.16]{GGS}). Choosing $p_2=0$ we obtain $(\nabla  \Delta u \cdot n)_{|\partial D} = h$ in $H^{-3/2}(\partial D)$; choosing $p_1=0$ it follows $(\Delta u)_{|\partial D}  = g$ in $H^{-1/2}(\partial D)$. 

In conclusion, 
in spite of the fact that the Neumann boundary value problem for the biharmonic operator does not satisfy the complementing condition, if  $f \in L^2(D)$, $g \in H^{-1/2}(\partial D)$ and $h \in H^{-3/2}(\partial D)$ we have proved that existence, uniqueness and continuous dependence on the data hold for the weak formulation \eqref{weekeqbih}, and we have also verified that, if the data satisfy the compatibuility condition \eqref{neccon}, the solution $u$ of \eqref{weekeqbih} is a (distributional) solution of \eqref{neubih}.

\begin{remark}\label{immerH2}
Looking at the proof just presented, we see that, under the assumption $\partial D \in C^4$, we have also proved the continuous immersion
$$
V_\sharp = L^2(\Delta;D) \cap {\cal H}^\bot \hookrightarrow H^2(D) \, .
$$
In fact, we have shown $V_\sharp = X$, with equivalence of the norms, and $X$ is a closed subspace of $H^2(D)$. Therefore, by the Rellich theorem the immersion $L^2(\Delta;D) \cap {\cal H}^\bot \hookrightarrow L^2(D)$ is compact; let us note again that for $L^2(\Delta;D)$ or   $L^2(\Delta;D) \cap {\cal H}$ this is not true  (see Valli \cite[Exercise 6.11]{AlbV}).
\end {remark}

\subsection{Reformulation as an elliptic system}

With the aim of devising a formulation that is more suitable for finite element numerical approximation, it can be interesting to rewrite the Neumann problem for the biharmonic operator as a system of two second order equations. In this way one can avoid the rather cumbersome need of higher order continuity of the discrete functions across the common boundaries of the elements.

We need to assume that the boundary data are more regular, precisely, we require $g \in H^{1/2} (\partial D)$ and $h \in H^{-1/2} (\partial D)$, and we start from the solution $u \in H^2(D)$ of problem \eqref{weekeqbih}. As a first step we set $\sigma = \Delta u$. From this we have $\Delta \sigma = f \in L^2(D)$ in $D$ and $\sigma_{|\partial D} = g \in H^{1/2} (\partial D)$ on $\partial D$, so that $\sigma \in H^1(D)$; we also have $\nabla \sigma \cdot n = h$ on $\partial D$, as $\nabla \sigma \in H(\text{div};D)$, thus its normal trace has a meaning in $H^{-1/2}(\partial D)$ (see, e.g., Girault and Raviart \cite[Theor. 2.5]{GirRav}). Then we extend $g$ inside $D$ by $\psi_g \in H^1(D)$, and finally we set $\lambda = \sigma - \psi_g$, which belongs to $H^1_0(D)$.

By integration by parts we obtain
$$
\begin{array}{ll}
0 &= \int_D \sigma \, \tau \, dx - \int_D \Delta u \, \tau \, dx \\ [4pt] &= \int_D \sigma \, \tau \, dx + \int_D \nabla u \cdot \nabla \tau \, dx - \int_{\partial D} \nabla u \cdot n \, \tau \, dS_x  \\[4pt]&= \int_D \lambda \, \tau \, dx + \int_D \psi_g \, \tau \, dx + \int_D \nabla u \cdot \nabla \tau \, dx \quad \forall \ \tau \in H^1_0(D)
\end{array}
$$
and
$$
\begin{array}{ll}
0 &= \int_D f \, v \, dx -\int_D \Delta \sigma \, v \, dx =  \int_D f \, v \, dx + \int_D \nabla \sigma \cdot \nabla v \, dx - \langle \nabla \sigma \cdot n,v\rangle_{-\frac12} \\[4pt]
&=  \int_D f \, v \, dx + \int_D \nabla \lambda \cdot \nabla v \, dx + \int_D \nabla \psi_g \cdot \nabla v \, dx - \langle h,v\rangle_{-\frac12} \quad \forall \ v \in H^1(D) \, .
\end{array}
$$
Thus $u \in H^1(D) \cap {\cal H}^\bot$ and $\lambda \in H^1_0(D)$ are a solution to the system
\begin{equation}\label{mixedpb}
\begin{array}{ll}
\int_D \lambda \, \tau \, dx + \int_D \nabla u \cdot \nabla \tau \, dx =  - \int_D \psi_g \, \tau \, dx  \quad \forall \ \tau \in H^1_0(D)\\ [6pt]
\int_D \nabla \lambda \cdot \nabla v \, dx \\ [4pt]\qquad = - \int_D f \, v \, dx - \int_D \nabla \psi_g \cdot \nabla v \, dx + \langle h,v\rangle_{-\frac12} \quad \forall \ v \in H^1(D) \cap {\cal H}^\bot\, .
\end{array}
\end{equation}
We want to show that uniqueness holds for this system. Assume that its right hand sides vanish; then taking $\tau = \lambda$ and $v = u$ we find $\int_D \nabla \lambda \cdot \nabla u \, dx = 0$, therefore $\int_D \lambda^2 \, dx = 0$ and $\lambda = 0$ in $D$. We are thus left with $\int_D \nabla u \cdot \nabla \tau \, dx = 0$ for all $\tau \in H^1_0(D)$, which says $\Delta u = 0$ in $D$. Hence $u \in {\cal H}$, and from $u \in {\cal H}^\bot$ we conclude $u = 0$.

Thus the solution $u$ of problem \eqref{weekeqbih}, together with $\lambda = \Delta u - \psi_g$, is a solution of problem \eqref{mixedpb}: since the solution of this last problem is unique, by solving \eqref{mixedpb} we find the solution $u$ of \eqref{weekeqbih}.

However, system \eqref{mixedpb} is not yet satisfactory from the point of view of numerical approximation, as the space $H^1(D) \cap {\cal H}^\bot$ has the orthogonality constraint, which is not easy to be handled with finite elements.

Thus we modify the previous procedure. Start as before from the solution $u \in H^2(D)$ of problem \eqref{weekeqbih}, and construct the harmonic function $\eta_u \in H^2(D)$ solution of 
$$
\left\{\begin{array}{ll}
\Delta \eta_u = 0 &\hbox{\rm in} \ D\\ 
\eta_{u|\partial D} = u_{|\partial D} &\hbox{\rm on} \ \partial D \, .
\end{array}
\right.
$$
Thus $s = u - \eta_u$ is another solution of the Neumann problem for the operator $\Delta^2$, and moreover $s_{|\partial D} = 0$ on $\partial D$. From now on we proceed as before:  we define $\sigma = \Delta s$, $\lambda = \sigma - \psi_g$ (so that, in particular, $(\nabla \sigma \cdot n)_{|\partial D} = h$ and $\lambda_{|\partial D} = 0$ on $\partial D$) and obtain the system 
$$
\begin{array}{ll}
\int_D \lambda \, \tau \, dx + \int_D \nabla s \cdot \nabla \tau \, dx =  - \int_D \psi_g \, \tau \, dx  \quad \forall \ \tau \in H^1_0(D)\\ [6pt]
\int_D \nabla \lambda \cdot \nabla v \, dx \\ [4pt]\qquad = - \int_D f \, v \, dx - \int_D \nabla \psi_g \cdot \nabla v \, dx + \langle h,v\rangle_{-\frac12} \quad \forall \ v \in H^1(D) \, .
\end{array}
$$
Since $s \in H^1_0(D)$, we are even led to conclude that $s$ and $\lambda \in H^1_0(D)$ are solutions of 
\begin{equation}\label{mixedpb2}
\begin{array}{ll}
\int_D \lambda \, \tau \, dx + \int_D \nabla s \cdot \nabla \tau \, dx =  - \int_D \psi_g \, \tau \, dx  \quad \forall \ \tau \in H^1_0(D)\\ [6pt]
\int_D \nabla \lambda \cdot \nabla v \, dx \\ [4pt]\qquad = - \int_D f \, v \, dx - \int_D \nabla \psi_g \cdot \nabla v \, dx  \quad \forall \ v \in H^1_0(D) \, .
\end{array}
\end{equation}
(which is a triangular system!).

The solution $(s,\lambda)$ of this problem is unique: taking $\tau = \lambda$ and $v=s$ it follows $\int_D \nabla s \cdot \nabla \lambda \, dx = 0$, then $\int_D \lambda^2 \, dx = 0$ and $\lambda = 0$ in $D$. Moreover, taking $\tau = s$ (that is now possible, as $s \in H^1_0(D)$) we find $\nabla s = 0$ in $D$ and, from the homogeneous boundary condition, $s = 0$ in $D$. 

As before, the conclusion is that, if we find the solution $(s,\lambda)$ of \eqref{mixedpb2}, we have also found a solution $s$ of the Neumann boundary value problem for the $\Delta^2$ operator. This solution $s$ is the only one that satisfies $s_{|\partial D} =0$ on $\partial D$ (but no longer we have $s \in {\cal H}^\bot$). In other word, we have selected a different solution of the Neumann problem.

If you are worried by the fact that the boundary datum $h$ does not appear in \eqref{mixedpb2}, consider that we still have available the compatibility condition \eqref{neccon}. A direct computation shows that for each $\eta \in H^2(D) \cap {\cal H}$ the solution $s$ satisfies
$$
\int_D f \, \eta \, dx = \int_D \eta \, \Delta^2 s  \, dx = \langle\nabla \Delta s \cdot n,\eta\rangle_{-\frac32} - \langle \Delta s, \nabla \eta \cdot n\rangle_{-\frac12} \, .
$$
Inserting this result in \eqref{neccon} and recalling that $(\Delta s)_{|\partial D} = g$ on $\partial D$ it follows
$$
\langle\nabla \Delta s \cdot n,\eta\rangle_{-\frac32} = \langle h,\eta\rangle_{-\frac32} \, .
$$
Taking $\xi \in H^{3/2}(\partial D)$ and finding the solution $\eta_\xi\in H^2(D) \cap {\cal H}$ of 
$$
\left\{\begin{array}{ll}
\Delta \eta_\xi = 0 &\hbox{\rm in} \ D \\ 
\eta_{\xi|\partial D} = \xi &\hbox{\rm on} \ \partial D \, ,
\end{array}
\right.
$$
we conclude that $\langle\nabla \Delta s \cdot n - h,\xi\rangle_{-\frac32} = 0$ for each $\xi \in H^{3/2}(\partial D)$, thus $\nabla \Delta s \cdot n = h$ on $\partial D$ (as elements of $H^{-3/2}(\partial D)$).

In conclusion, the variational formulation \eqref{mixedpb2} is the best suited for finite element numerical approximation: it consists in two homogeneous Dirichlet boundary value problems for the Laplace operator, to be solved in cascade.

\subsection{An additional result: well-posedness of an overdetermined Poisson problem}

The procedures we employed in the preceding sections are essentially the same that lead to the following basic existence theorem for an overdetermined Poisson problem. Let us mention that we were not previously aware of this result, but by a search on Google we have found a quite similar proof in 

\noindent ``https://mathoverflow.net/questions/316421/overdetermined-poisson-equation", 

\noindent where a simple note by Mateusz Kwa\'snicki was posted in 2018. 

\begin{theorem}\label{overth}
Assume that $D \subset \mathbb{R}^{n}$ is a bounded, connected, open set, with  a boundary $\partial D \in C^2$ and that $p \in L^2(D)$. The overdetermined Poisson boundary value problem
\begin{equation}\label{over}
\left\{\begin{array}{ll}
\Delta U = p&\hbox{\rm in} \ D\\ 
U_{|\partial D} = 0 &\hbox{\rm on} \ \partial D \\
(\nabla U \cdot n)_{|\partial D} = 0 &\hbox{\rm on} \ \partial D
\end{array}
\right.
\end{equation}
has a (unique) solution $U \in H^2_0(D)$ if and only if $p \in {\cal H}^\bot$.
\end{theorem}

\begin{proof}
We have already seen that $p = \Delta U$ belongs to ${\cal H}^\bot$ if $U \in H^2_0(D)$.

Vice versa, take $p \in {\cal H}^\bot$ and consider the variational problem:
\begin{equation}\label{lasteq}
\begin{array}{ll}
&\text{find } W \in H^2_0(D) \, : \\ [4pt]
&\qquad \int_D \Delta W \, \Delta \mu \, dx = \int_D p \, \Delta \mu \, dx  \\[4pt] 
&\text{for all} \ \mu \in H^2_0(D) \, .
\end{array}
\end{equation}
By the regularity results for elliptic problems we have $\|\mu\|_{H^2(D)} \le C \|\Delta \mu\|_{L^2(D)}$ for each $\mu \in H^2(D) \cap H^1_0(D)$. Therefore a first consequence is that uniqueness holds for problem \eqref{over}; moreover, the bilinear form at the left hand side of problem \eqref{lasteq} is coercive in $H^2(D)$, hence we have a unique solution $W \in H^2_0(D)$  by the Lax--Milgram theorem. We thus have $(\Delta W - p) \in {\cal H}$. On the other hand, from $W \in H^2_0(D)$ we have $\Delta W \in {\cal H}^\bot$, therefore $(\Delta W - p) \in {\cal H}^\bot$. In conclusion, $\Delta W = p$ in $D$, hence $W$ is a solution of problem \eqref{over}.
\end{proof}

In a similar way,  we also have
\begin{theorem}
Assume that $D \subset \mathbb{R}^{n}$ is a bounded, connected, open set, with  a boundary $\partial D \in C^4$ and that $p \in L^2(D)$. The overdetermined fourth order boundary value problem
\begin{equation}\label{over2}
\left\{\begin{array}{ll}
\Delta^2 V = p&\hbox{\rm in} \ D\\ 
V_{|\partial D} = 0 &\hbox{\rm on} \ \partial D \\
(\Delta V)_{|\partial D} = 0 &\hbox{\rm on} \ \partial D \\
(\nabla \Delta V \cdot n)_{|\partial D} = 0 &\hbox{\rm on} \ \partial D
\end{array}
\right.
\end{equation}
has a (unique) solution $V \in H^4(D) \cap H^1_0(D)$ if and only if $p \in {\cal H}^\bot$.
\end{theorem}

\begin{proof}
Since $\Delta V \in H^2_0(D)$, we have as before $p = \Delta(\Delta V) \in {\cal H}^\bot$.

Vice versa, for $p \in {\cal H}^\bot$ take the solution $U \in H^2_0(D)$ of problem \eqref{over}. Then consider the solution $V \in H^1_0(D)$ of the problem
$$
\left\{\begin{array}{ll}
\Delta V = U &\hbox{\rm in} \ D\\ 
V_{|\partial D} = 0 &\hbox{\rm on} \ \partial D \, .
\end{array}
\right.
$$
Since $\partial D \in C^4$ by the elliptic regularity results it follows $V \in H^4(D) \cap H^1_0(D)$, and it is readily shown that $V$ is a solution to \eqref{over2}.

Concerning uniqueness, setting $p=0$ in \eqref{over2} by integration by parts we have 
$$
0 = \int_D V \, \Delta^2 V \, dx = \int_D \Delta V \, \Delta V \, dx \, ,
$$
thus $\Delta V = 0$ in $D$ and from the boundary condition $V_{|\partial D} = 0$ on $\partial D$ it follows $V=0$ in $D$.
\end{proof}

\medskip

{\sl Note added in proof.}
After having completed this work, we got to know that the use of the space $X = \Delta[H^4(D) \cap H^2_0(D)]$ had been already proposed by Provenzano \cite{prov} for studying an eigenvalue problem for the biharmonic operator with the Neumann boundary conditions. It seems rather clear that a  suitable rephrasing of the results there obtained would have led to an existence and uniqueness theorem similar to the one presented here.

\medskip

{\bf Acknowledgement.} It is a pleasure to thank Filippo Gazzola, Pier Do\-me\-ni\-co Lamberti and Dirk Pauly  for some useful conversations about boundary value problems for the biharmonic operator.
I also acknowledge the support of INdAM-GNCS.

%%%%%%%%%%%%%%%%%%%%%%%%%%%%%%%%%%%%
\bibliographystyle{plain}
\bibliography{biharm}

\begin{thebibliography}{10}

\bibitem{ADN}
S.~Agmon, A.~Douglis, and L.~Nirenberg.
\newblock Estimates near the boundary for solutions of elliptic partial
  differential equations satisfying general boundary conditions. {I}.
\newblock {\em Comm. Pure Appl. Math.}, 12:623--727, 1959.

\bibitem{beg}
H.~Begehr.
\newblock Biharmonic {G}reen functions.
\newblock {\em Matematiche (Catania)}, 61:395--405, 2006.

\bibitem{couhil}
R.~Courant and D.~Hilbert.
\newblock {\em Methods of mathematical physics. Vol. 1}.
\newblock Interscience Publishers, Inc. New York, 1953.

\bibitem{couhil2}
R.~Courant and D.~Hilbert.
\newblock {\em Methods of mathematical physics. Vol. 2}.
\newblock Interscience Publishers, Inc. New York, 1962.

\bibitem{GGS}
F.~Gazzola, H.-C. Grunau, and G.~Sweers.
\newblock {\em Polyharmonic boundary value problems}.
\newblock Springer, Berlin, 2010.

\bibitem{GirRav}
V.~Girault and P.-A. Raviart.
\newblock {\em Finite element methods for {N}avier-{S}tokes equations}.
\newblock Springer, Berlin, 1986.

\bibitem{LioMag}
J.-L. Lions and E.~Magenes.
\newblock {\em Non-homogeneous boundary value problems and applications. Vol.
  1}.
\newblock Springer, Berlin, 1972.

\bibitem{NegMon}
P.V. Negr\'{o}n-Marrero and E.~Montes-Pizarro.
\newblock The complementing condition and its role in a bifurcation theory
  applicable to nonlinear elasticity.
\newblock {\em New York J. Math.}, 17A:245--265, 2011.

\bibitem{PauSch}
D.~Pauly and M.~Schomburg.
\newblock Hilbert complexes with mixed boundary conditions. {P}art 3:
  biharmonic complexes.
\newblock Technical report, arxiv:2207.11778v2, 2023.

\bibitem{PauZul}
D.~Pauly and W.~Zulehner.
\newblock The div{D}iv-complex and applications to biharmonic equations.
\newblock {\em Appl. Anal.}, 99:1579--1630, 2020.

\bibitem{prov}
L.~Provenzano.
\newblock A note on the {N}eumann eigenvalues of the biharmonic operator.
\newblock {\em Math. Methods Appl. Sci.}, 41:1005--1012, 2018.

\bibitem{RenRog}
M.~Renardy and R.~C. Rogers.
\newblock {\em An introduction to partial differential equations}.
\newblock Springer, New York, 2nd edition, 2004.

\bibitem{AlbV}
A.~Valli.
\newblock {\em A compact course on linear PDEs.}
\newblock Springer, Cham, 2nd edition, 2023.

\bibitem{ver}
G.C. Verchota.
\newblock The biharmonic {N}eumann problem in {L}ipschitz domains.
\newblock {\em Acta Math.}, 194:217--279, 2005.

\bibitem{Wloka}
J.~Wloka.
\newblock {\em Partial differential equations}.
\newblock Cambridge University Press, Cambridge, 1987.

\end{thebibliography}

\end{document}